\documentclass[11pt]{amsart}
\usepackage{color}
\usepackage{amssymb}
\theoremstyle{plain}
\newtheorem{Thm}{Theorem}

\newtheorem{Pro}[Thm]{Proposition}

\begin{document}

\title[Liouville theorems for Ricci shrinkers]
{Liouville theorems, Volume growth, and volume comparison for Ricci shrinkers}

\author{Li Ma}

\address{ Prof. Dr.Li Ma\\
Department of mathematics \\
Henan Normal university \\
Xinxiang, 453007 \\
China
}
\email{lma@tsinghua.edu.cn}

\thanks{The research is partially supported by the National Natural Science
Foundation of China No.11271111. \\
 This work was done when the
author is visiting Minnesota University in September 2016 and he would
like to thank Prof. Jiaping Wang for helpful discussion and the hospitality of the school of Mathematics in Minnesota.
}

\begin{abstract}
In this paper, we study volume growth, Liouville theorem and the local gradient estimate for $f$-harmonic functions, and volume comparison property of unit balls in complete noncompact gradient Ricci shrinkers.
We also study integral properties of f-harmonic functions and harmonic functions on such manifolds.

{ \textbf{Mathematics Subject Classification 2000}: 53C21,
53C44, 53C21, 53C25}

{ \textbf{Keywords}: Ricci shrinkers, f-harmonic functions, Liouville theorem, volume
comparison}
\end{abstract}

 \maketitle

\section{introduction}
In the study of Ricci flow on a compact Riemannian manifold, because of its complicated nonlinearity, one meets singularities of the flow in finite time. After blowing up, one expects to get self-similar Ricci shrinkers or steady Ricci solitons (\cite{H}). In the case of type I singularity in dimension three, one gets nontrivial gradient Ricci shrinkers via the use Ivey-Hamilton pinching estimate and the classification of this type of self-similar Ricci shrinkers has been done by G.Perelman \cite{P}. In case of type I singularity of dimension four, A.Naber \cite{N} has showed that one gets a gradient Ricci shrinker and non-trivial property of this Ricci shrinker has been done later by others. These solitons can be considered as special examples of weighted Riemannian manifolds or metric measure spaces (\cite{BE} \cite{WW} \cite{LV} \cite{C} \cite{CLN}\cite{L}\cite{lo}\cite{MW1}\cite{Ma} \cite{S1,S2}\cite{Y} \cite{MW2} and \cite{MW3}). Because of the importance of four dimensional Ricci shrinkers, many people study various properties of them (just cite a few for references \cite{Ma}\cite{CZ} \cite{C1,C2}\cite{CN} \cite{MW1} and \cite{HM}).
In this paper, we study three questions about Ricci shrinking and steady solitons, in particular for Ricci shrinkers which are the complete Riemannian manifold $(M,g)$ such that $Ric_f:=Ric+D^2f=\lambda g$ on $M$, where $Ric$ is the Ricci curvature of $(M,g)$, $f:M\to R$ is a smooth function in $M$, and $\lambda>0$ is a constant. One question is about the volume comparison of unit balls on Ricci shrinkers. The other two are about f-harmonic functions and harmonic function with finite energy.  The volume comparison of unit balls is an important step to understand the volume growth of geodesic balls in the gradient Ricci shrinkers. The Liouville theorems for f-harmonic and harmonic functions with finite energy are important to understand the connectivity at infinity about gradient Ricci solitons, see \cite{MW1}.

We have the following new results. The first one is the volume comparison of unit balls at any point $x\in M$ with a unit ball at a fixed point $p\in M$, which is about the injectivity radius decay from the point $p$ to the point $x$ and it is important to understand the topology of the underlying manifold at infinity.

\begin{Thm}\label{wangjp} On the complete noncompact gradient Ricci shrinker $(M,g,f)$ with Ricci curvature bounded below by $-(n-1)k^2$ for some constant $k\geq 0$, we have
$$
Vol(B_x(1)\geq exp(-\sqrt{c(n-1)R})Vol(B_p(1)).
$$
where $c_0$ is a uniform constant which does not depends on $x$ and $R=d(p,x)>R_0$ for some uniform constant $R_0>1$.
\end{Thm}

With the extra condition that the Ricci curvature has a lower bound, this improves the result of Lemma 5.2 in \cite{MW1}.

By the well-known argument, we know that there is no nontrivial positive $f$-harmonic function on a gradient Ricci shrinker. In fact there is no non-constant $f$-super-harmonic function $u$ ($\Delta_fu:=\Delta u-\nabla f.\nabla u\leq 0$) on the complete Riemannian manifold $(M,g)$ with $Ric_f\geq \frac{\lambda}g$ on $M$ for some positive constant $\lambda$. The process of proving this is below. Recall that the weighted volume of $(M,g)$ is finite \cite{M}, i.e., $V_f(M):=\int_M e^{-f}dv_g<\infty$. Assume $u>0$ is such a positive $f$-super-harmonic function on $M$. Let $v=\log u$. Then we have
$$
\Delta_f v=\frac{\Delta_f u}{u}-|\nabla v|^2\leq -|\nabla v|^2.
$$
Then we have, for any cut-off function $\phi\geq0$ on the ball $B_p(2R)\subset M$ with $\phi=1$ in $B_p(R)$,
$$
\int_M|\nabla v|^2\phi^2e^{-f}dv_g\leq 2\int_M \phi\nabla\phi.\nabla ve^{-f}dv_g.
$$
By the Cauchy-Schwartz inequality we get that
$$
\int_M|\nabla v|^2\phi^2e^{-f}dv_g\leq 4\int_M |\nabla\phi|^2e^{-f}dv_g\leq \frac{16}{R^2}V_f(M)\to 0
$$
as $R\to\infty$. Hence, $|\nabla v|^2=0$ in $M$, which implies that $u$ is a constant on $M$.

With this regard, it is interesting to understand the gradient estimate of $f$-harmonic function defined locally in part of $M$ (\cite{CLN}). We have the following local gradient estimate for positive $f$-harmonic functions in the ball $B_p(2R)$.
\begin{Thm}\label{wangjp1}
Let $(M,g,f)$ be a complete noncompact Ricci shrinker such that $Ric_f:=Ric+D^2f=\frac12 g$ with the Ricci curvature bounded below by $-(n-1)k^2$ for some constant $k\geq 0$.
For any positive $f$-harmonic function $u$ defined in the ball $B_p(2R)$, we have
$$
\frac{u(x)}{u(p)}\leq Cexp(Cd(x,p)^2), \ \ x\in B_p(R)
$$
for some uniform constant $C>0$, where $R>1$.
\end{Thm}

We are now trying to find another kind of Liouville theorem for $f$-harmonic functions on a gradient Ricci shrinker. We show that as a direct consequence of Bochner type formula (see \cite{MD}), we have the following Liouville type theorem.

\begin{Thm}\label{wangjp1+}
Let $(M,g,f)$ be a complete noncompact Ricci shrinker such that $Ric_f:=Ric+D^2f\geq \frac12 g$.
There is no nontrivial $f$-harmonic function $u$ defined in $(M,g))$ with weighted finite energy, i.e.,
$$
\int_M(|\nabla u|^2)e^{-f}dv_g<\infty.
$$
\end{Thm}
The proof of this result is given in section \ref{sect3}.

\begin{Pro}\label{wangjp2} Fix any $p\in M$. Assume that $(M,g)$ satisfies that $Ric_f\geq \frac12g$
 with $$|\nabla f(x)|\leq \alpha d(x,p)+b $$  for some unform constants $\alpha\geq 0$,$\lambda \geq0$, and $b>0$, and has non-negative scalar curvature, i.e., $R\geq 0$. Then for any harmonic function $u$ with finite energy
 $$
 \int_M|\nabla u|^2<\infty,
 $$ we have the integral inequality
$$
\int_M|\nabla^2 u|^2+\frac{1}{2}\int_M R|\nabla u|^2\leq \frac{n\lambda}{2}\int_M|\nabla u|^2.
$$
\end{Pro}
We now give a few remarks.

1. In \cite{N}, A.Naber proves that for the weighted smooth metric space $(M,g,f)$ satisfying $Ric_f\leq\frac12 g$ and $|Ric|\leq C$, there exists $\alpha>0$ such that if $\Delta_fu:=\Delta u-\nabla f.\nabla u=0$ on $M$ with $u(x)|\leq Aexp(\alpha d(x,p)^2$ for some $A>0$ and $p\in M$, then $u$ is a constant.

2. In \cite{MS}, Munteanu and Sesum prove that for the gradient shrinking Kaehler-Ricci soliton, if the harmonic function $u$ has finite energy, i.e., $\int_M|\nabla u|^2<\infty$, then $u$ is a constant. As a consequence of this result, they can show that such a manifold has at most one non-parabolic end (see \cite{MS} for the definition of non-parabolic end). In the earlier work \cite{MW3}, Munteanu and Wang have proved that on a  weighted smooth metric space $(M,g,f)$ satisfying $Ric_f\leq 0$ and $f$ is a bounded function, any sublinear growth $f$-harmonic function on $M$ must be a constant.

3. Some consequences of Proposition \ref{wangjp2} are given in section \ref{sect4}.

Here is the plan of the paper. In section \ref{sect1}, consider the mean curvature bounds of spheres centered at any point. We study
the volume comparison of unit balls in section \ref{sect2}. We  prove Theorem \ref{wangjp1} and Theorem \ref{wangjp1+} in section \ref{sect3}. In the last section we consider integral properties and prove Proposition \ref{wangjp2}.

\section{mean curvature bounds of spheres centered at any point}\label{sect1}

Note that the Ricci curvature lower bound gives us the upper bound for the Hessian matrix that  $D^2f\leq \frac12+\delta$ for some $\delta\geq 0$ on the Ricci shrinker  $(M,g,f)$ with the normalized condition
$$
Ric_f=\frac12 g.
$$
 Let $x\in M$.
We want to bound the mean curvature $m_x(s)$ of the sphere
$\partial B_x(s)$ of radius $s>0$ in term of $\delta$.
Our result is the following.

\begin{Pro} Assume that the weighted smooth metric space $(M,g,f)$ satisfying $Ric_f\geq\frac12 g$ and $D^2f\leq \frac12+\delta$ for some $\delta\geq 0$. For any $x\in M$, let $m(r)$ be the mean curvature of the sphere $\partial B_x(r)$ in $M$. Then
$$
m(r)\leq \frac{n-1}{r}+\frac{\delta r}{3}
$$
for any $r\geq 1/2$ and $Vol(B_x(r))\leq Cexp(\delta r^2)$ for some uniform constant $C>0$.
\end{Pro}

\begin{proof}
Take any point $x\in M$ and express the volume form in the geodesic polar coordinates centered at $x$ as
$$
dV|_{exp_x(r\xi)}=J(x,r,\xi)drd\xi
$$
for $r>0$ and $\xi\in S_xM$, a unit tangent vector at $x$. For any $y\in M$, we let $R=d(y,x)$ and omit the dependence of the geometric quantities on $\xi$.
  We may assume that $R=d(y,x)>1$ and let $\gamma(s)$ be the minimizing geodesic starting from $x$ such that $\gamma(0)=x$ and $\gamma(T)=y$, where $T\in [R-1,R]$. Recall that along the minimizing geodesic curve $\gamma(r)$,
 $$
 m'(r)+\frac{1}{n-1}m^2(r)+Ric(\partial_r,\partial_r)\leq 0,
 $$
where $m=m(r)=\frac{d}{dr}(\log J)(r)$.
Using the Ricci soliton equation $Ric_f=\frac{1}{2}$ we immediately obtain that
$$
 m'(r)+\frac{1}{n-1}m^2(r)\leq -\frac{1}{2}+f''(r).
 $$

We now use test function to give a upper bound of $m(r)$. For any $k\geq 2$, multiplying the above differential inequality by $r^k$
 and integrating from $r=0$ to $r=t$, we obtain
 $$
 \int_0^t m'(r)r^kdr+ \int_0^t\frac{1}{n-1}m^2(r)r^kdr\leq -\frac{t^{k+1}}{2(k+1)}+ \int_0^tf''(r)r^kdr.
 $$

 Integrating the first term by part and making square for the second term, we have
 $$
 m(t)t^k+ \frac{1}{n-1}\int_0^t(m(r)r^{k/2}-(n-1)\frac{k}{2}r^{\frac k2-1} )^2dr
 $$
 $$\leq \frac{(n-1)k^2t^{k-1}}{4(k-1)}-\frac{t^{k+1}}{2(k+1)}+ \int_0^tf''(r)r^kdr.
 $$
 This implies that
 $$
 m(t)\leq \frac{(n-1)k^2}{4(k-1)t}-\frac{t}{2(k+1)}+ \frac{1}{t^k}\int_0^tf''(r)r^kdr.
 $$

 Using the assumption about the Hessian of $f$, we know that
 $f''(r)\leq \frac12+\delta$. By direct computation, we have
  $$
  m(t)\leq \frac{(n-1)k^2}{4(k-1)t}-\frac{t}{2(k+1)}+(\frac12+\delta)\frac{t}{k+1}.
  $$
  Choose $k=2$, we get
  $$
  m(t)\leq \frac{n-1}{t}+\frac{\delta t}{3},
  $$
  which give a upper bound in terms of the radius of the sphere. Integrating, for $t>1$, we have
  $$
  J(x,r,\xi)\leq C exp(\frac{\delta r^2}{6}),
  $$
  which implies that $Vol(B_x(r))\leq Cexp(\delta r^2)$ as we wanted.
\end{proof}

\section{volume comparison of unit balls}\label{sect2}
In section we prove Theorem \ref{wangjp}. We use the idea similar to the proof of Theorem 2.3 in \cite{MW1}.

We now give a proof of improved volume comparison of unit balls on the weighted Riemannian manifold of shrinking type.

\begin{proof} (of Theorem \ref{wangjp}).

Again, we take any point $x\in M$ and express the volume form in the geodesic polar coordinates centered at $x$ as
$$
dV|_{exp_x(r\xi)}=J(x,r,\xi)drd\xi
$$
for $r>0$ and $\xi\in S_xM$, a unit tangent vector at $x$. We let $R=d(p,x)$ and omit the dependence of the geometric quantities on $\xi$.
 Let $R=d(p,x)>1$ and let $\gamma(s)$ be the minimizing geodesic starting from $x$ such that $\gamma(0)=x$ and $\gamma(T)\in B_p(1)$ with $T\in [R-1,R+1]$. It is well-known that along the minimizing geodesic curve $\gamma$,
 $$
 m'(r)+\frac{1}{n-1}m^2(r)+Ric(\partial_r,\partial_r)\leq 0,
 $$
where $m=m(r)=\frac{d}{dr}(\log J)(r)$.
Using the Ricci soliton equation $Ric_f=\frac{1}{2}$ we immediately obtain that
$$
 m'(r)+\frac{1}{n-1}m^2(r)\leq -\frac{1}{2}+f''(r).
 $$
 Integrating this relation we get for $T\in {R-1,R-1}$, $s\in [1/2,1]$,
 $$
 m(T)+\frac{1}{n-1}\int_1^Tm^2(r)dr\leq -\frac{T-1}{2}+f'(T)-f'(1)+m(1).
 $$
Recall the following well-known fact that for $R>R_0>0$ very large, we have
$$
\frac{1}{2}(R-1)-c\leq f'(1)\leq \frac{1}{2}(R-1)+c
$$
and $|f'(T)|\leq c$ since $\gamma(T)\in B_p(1)$.
Here and everywhere in the proofs, $c$ and $R_0$ denote constants depending only on the dimension $n$ and $f(p)$.

Standard argument shows that there is a uniform constant $c_0>0$ such that $m(s)\leq c_0$ for $s\in [1/2,1]$. Then we have
$$
m(t)+\frac{1}{n-1}\int_1^tm^2(r)dr\leq c_0.
$$
By the Cauchy-Schwartz inequality we obtain that
\begin{equation}\label{mean}
m(T)+\frac{1}{(n-1)T}(\int_1^Tm(r)dr)^2<c,
\end{equation}
for $c>c_0$.

\textbf{Claim}: For any $r>1$,
\begin{equation}\label{mean1}
\int_1^Tm(r)dr\leq \sqrt{c(n-1)r}.
\end{equation}
In fact, let
$$
v(t)=\sqrt{c(n-1)t}-\int_1^tm(r)dr.
$$
Then
$$
v'(t)=\frac{\sqrt{c(n-1)}}{2\sqrt{r}}-m(t).
$$

Clearly $v(1)>0$ by choosing $c>c_0$. Suppose that $v$ is negative somewhere for $t>1$. Let $R>1$ be the first zero point of $v$, i.e.,
$v(R)=0$. Then by the choice of $R$, we have $v'(R)\leq 0$. That is,
$$
\int_1^Rm(r)dr=\sqrt{c(n-1)R}
$$
and
$$
m(R)\geq \frac{\sqrt{c(n-1)}}{2\sqrt{R}}.
$$
By direct computation we know that
$$
m(R)+\frac{1}{(n-1)R}(\int_1^Rm(r)dr)^2\geq \frac{\sqrt{c(n-1)}}{2\sqrt{R}}+c,
$$
which is a contradiction with (\ref{mean}).

The relation (\ref{mean1}) implies that
$$
\log J(x,T,\xi)/J(x,1,\xi)\leq \sqrt{c(n-1)T},
$$
and we have
$$
J(x,1,\xi)\geq exp(-\sqrt{c(n-1)R}) J(x,T,\xi).
$$
Integrating over the unit tangent vectors $\xi$ we get
$$
Area(\partial B_x(1)\geq exp(-\sqrt{c(n-1)R})Vol(B_p(1)),
$$
where $R=d(p,x)>R_0$. Similarly we have
$$
Area(\partial B_x(s)\geq exp(-\sqrt{c(n-1)R})Vol(B_p(1)),
$$
for any $s\in [1/2,1]$. Hence we have
$$
Vol(B_x(1)\geq exp(-\sqrt{c(n-1)R})Vol(B_p(1)).
$$
This is the desired result.
\end{proof}

\section{local gradient estimate for f-harmonic functions on Ricci shrinkers}\label{sect3}

We  prove Theorem \ref{wangjp1} and Theorem \ref{wangjp1+} in this section.

\begin{proof} (of Theorem \ref{wangjp1})
  Let $(M,g,f)$ be a complete noncompact Ricci shrinker such that $Ric_f=\frac12 g$. Define the drifting Laplacian by
  $$
  \Delta_f u=\Delta u-\nabla f.\nabla u.
  $$
Assume that $u>0$ be a f-harmonic function on $M$, i.e., $\Delta_f u=0$ on $M$.
Let $v=\log u$. Then
$$
v_j=\frac{u_j}{u}, \ \ \ v_{ij}=\frac{u_{ij}}{u}-|\nabla v|^2.
$$
Then
$$
\Delta_f v=-|\nabla v|^2.
$$
Recall that
$$
\frac{1}{2}\Delta_f |\nabla v|^2=|\nabla^2 v|^2+(\nabla v,\nabla \Delta_f v)+Ric_f(\nabla u,\nabla u).
$$
and
$$
|\nabla^2 v|^2\geq \frac1n (\Delta v)^2,
$$
Then we have
$$
\frac{1}{2}\Delta_f |\nabla v|^2\geq \frac1n (\Delta v)^2-(\nabla v,\nabla |\nabla v|^2)+\frac12|\nabla v|^2,
$$
which implies that
$$
\frac{1}{2}\Delta_f |\nabla v|^2\geq \frac1n (|\nabla v|^2-\nabla f.\nabla v)^2-(\nabla v,\nabla |\nabla v|^2)+\frac12|\nabla v|^2
$$
Fix $\epsilon>0$ small. Recall that
$$
(a-b)^2\geq \frac{1-\epsilon}{1+\epsilon}a^2-\frac 1\epsilon b^2.
$$
Then we have
\begin{equation}\label{bochner3}
\frac{1}{2}\Delta_f |\nabla v|^2\geq \frac{1-\epsilon}{n(1+\epsilon)}|\nabla v|^4-\frac{1}{n\epsilon}|\nabla f|^2|\nabla v|^2+\frac12|\nabla v|^2-(\nabla v,\nabla |\nabla v|^2).
\end{equation}

Let $\phi$ be a cut-off function on $[-2,2]$. Let $\eta=\phi(\sqrt{f}/R)$ for any $R>1$.
Note that
$$
\Delta_f\eta=\Delta\eta-\nabla f.\nabla \eta=\Delta\eta-\nabla f.\nabla\sqrt{f}\frac{\eta'}{R},
$$
where $\frac{|\nabla f|}{R}\leq C$ and $|\nabla \sqrt{f}|\leq 1$.

Define $Q=\eta |\nabla v|^2$.
Notice that
$$
|\nabla f.\nabla \eta|\leq C, \ \ \  (\Delta_f\eta-\frac{2|\nabla \eta|^2}{\eta})\geq -C.
$$
At the maximum point $x_0$ of $Q$, we have
$$
\Delta_f Q\leq 0, \ \ \nabla Q=0.
$$
Note that at $x_0$,
$$
\nabla\eta |\nabla v|^2=-\eta\nabla |\nabla v|^2,
$$
and by
$$
0\geq \Delta_f Q=\Delta_f\eta.|\nabla v|^2 +2\nabla |\nabla v|^2.\nabla \eta +\eta\Delta_f |\nabla v|^2,
$$
we have
$$
(\Delta_f\eta-\frac{2|\nabla \eta|^2}{\eta})|\nabla v|^2+\eta \Delta_f |\nabla v|^2\leq 0.
$$
Write by $C(\eta)=\frac12£¨\Delta_f\eta-\frac{2|\nabla \eta|^2}{\eta}£©$ and $C_\epsilon=\frac{1-\epsilon}{n(1+\epsilon)}$. Then by (\ref{bochner3}) we have
$$
C(\eta)|\nabla v|^2+\eta[C_\epsilon|\nabla v|^4-\frac{1}{n\epsilon}|\nabla f|^2|\nabla v|^2+\frac12|\nabla v|^2
-(\nabla v,\nabla |\nabla v|^2)]\leq 0.
$$
Then we have
$$
C(\eta)Q+C_\epsilon Q^2-\frac{1}{n\epsilon}|\nabla f|^2Q+\frac12 Q+ \nabla v.\nabla\eta Q\leq 0.
$$
Note that
$$
|\nabla v.\nabla\eta|Q\leq (\frac{C_\epsilon}{2}Q+ \frac{1}{2C_\epsilon}\frac{|\nabla\eta|^2 }{\eta})Q.
$$
Then we have
$$
\frac12C_\epsilon Q+\frac12\leq\frac{1}{n\epsilon}|\nabla f|^2 -C(\eta).
$$
By this we get at the maximum point $x_0$,
$$
Q\leq C_\epsilon R^2
$$
for $R>1$. Hence we get that on $B_R$, $|\nabla v|\leq C_\epsilon R $. This implies that
$$
|\nabla v|(x)\leq Cd(x,p).
$$
for $d(x,p)>1$. Hence, we have the gradient estimate for $u>0$ on $M$,
$$
\frac{|\nabla u|}{u}(x)\leq C(d(x,p)+1).
$$
Take any minimizing curve $\gamma(s)$ from $p$ to $x$, we integrate along $\gamma$ to get
$$
\frac{u(x)}{u(p)}\leq Cexp(Cd(x,p)^2).
$$
This completes the proof of Theorem \ref{wangjp1}.
\end{proof}

 We hope that we can use the Cacciopolli argument (see the proof of Proposition 8.1 in Naber's paper \cite{N} and the use of Lemma 2.2 is replaced by Proposition 4.2 in \cite{MW2} ) to conclude that with some decay assumption such that finite energy, a $f$ harmonic function
$u$ is a constant function on $M$. However, we have a simpler proof of this result below.

\begin{proof}(of Theorem \ref{wangjp1+}). Recall that the Bochner formula for the harmonic function $u:M\to R$,
$$
\frac{1}{2}\Delta_f |\nabla u|^2=|\nabla^2 u|^2+Rc_f(\nabla u,\nabla u).
$$
By our assumption that $Ric_f\geq\frac12 g$, we know that
$$
\frac{1}{2}\Delta_f |\nabla u|^2\geq |\nabla^2 u|^2+\frac{1}{2}|\nabla u|^2.
$$
Let $\phi$ be the standard cut-off function on $B_p(2R)$ and let $dm=exp(-f)dv_g$. Then we have
$$
\int_M(|\nabla^2 u|^2+\frac{1}{2}|\nabla u|^2)\phi dm\leq \int_M(\frac{1}{2}\Delta_f\phi) |\nabla u|^2dm.
$$
The right side is going to zero as $R\to\infty$. Hence we have
$$
\int_M(|\nabla^2 u|^2+\frac{1}{2}|\nabla u|^2) dm=0,
$$
which implies that $u$ is a constant.
\end{proof}

We now consider the volume growth of geodesic balls in manifolds with density and we show that for $(M,g,e^{-f}dv)$ being a complete smooth metric measure space of dimension $n$ with $Ric_f\geq \frac12$, $|\nabla f|\leq f$, and also with both Ricci curvature bound above and $\Delta f$ bounded from above, the volume growth of geodesic balls is in polynomial order.

\begin{Pro}
Let $(M,g,e^{-f}dv)$ be a complete smooth metric measure space of dimension $n$. Assume that
$Ric_f\geq \frac12$, $|\nabla f|\leq f$. Assume further that $\Delta f\leq K$ and $Ric\leq K$ for some constant $K>0$. Then for any $p\in M$, the volume growth
of geodesic balls $B_p(r)$ are of polynomial order,i.e., there is a uniform constant $C>0$ such that
$$
V(r)\leq Cr^{2K}
.$$
\end{Pro}

\begin{proof}
Recall that under the conditions $Ric_f\geq \frac12$ and $|\nabla f|^2\leq f$, there are two constants $r_0>0$ and $a$ depending only on $n$ and $f(p)$ such that
\begin{equation}\label{cao1}
(\frac12 d(x,p)-a)^2\leq f(x)\leq (\frac12 d(x,p)+a)^2.
\end{equation}
This is from Proposition 4.2 in the interesting paper \cite{MW1}. By this we know that
$|\nabla f(x)|\leq \frac12 d(x,p)+a$. We may assume that $d(x,p)>2$. Consider any minimizing normal geodei=sic $\gamma(s)$, $0\leq s\leq r:=d(x,p)$ starting from $\gamma(0)=p$ to
$\gamma(r)=x$. Let $X=\dot{\gamma(s)}$. By the second variation formula of arc length we know that
$$
\int_0^r\phi^2Ric(X,X)ds\leq (n-1)\int_0^r|\dot{\phi}(s)|^2ds
$$
for any $\phi\in C_0^{1-}([0,r])$. Let $\phi(s)=s$ on $[0,1]$, $\phi(s)=r-s$ on $[r-1,r]$, and $\phi(s)=1$ on $[1,r-1]$.
Then we have
$$
\int_0^rRic(X,X)ds=\int \int_0^r\phi^2Ric(X,X)ds+\int_0^r(1-\phi^2)Ric(X,X)ds.
$$
Then we derive using $Ric\leq K$ (similar the proof before (2.8) in \cite{CZ}), we have
$$
\int_0^rRic(X,X)ds\leq 2(n-1)+2K.
$$
Since
$$
\nabla_X\dot{f}=\nabla^2f(X,X)\geq \frac12-Ric(X,X),
$$
Integrating it from $0$ to $r$, we get
\begin{equation}\label{cao2}
\dot{f}(r)=\frac12r-\int_0^2Ric(X,X)ds\geq \frac12 r-c
\end{equation}
for some constant $c$ depending only on $K$, $n$ and $f(p)$. Hence,
$$|\nabla f|(x)\geq \dot{f}(r)\geq \frac12 d(x,p)-c.$$

Define
$$
\rho(x)=2\sqrt{f(x)}.
$$
Then,
$$
|\nabla \rho|=\frac{|\nabla f|}{\sqrt{f}}\leq 1.
$$
Let for $r>0$ large,
$$
D(r)=\{x\in M; \rho(x)\leq r\}, \ \ V(r)=Vol(D(r)).
$$
As in \cite{CZ}, by the co-area formula we have
$$
V(r)=\int_0^rds\int_{\partial D(r)}\frac{1}{|\nabla \rho|}dA
$$
and
$$
V'(r)=\int_{\partial D(r)}\frac{1}{|\nabla \rho|}dA=\frac{r}{2}\int_{\partial D(r)}\frac{1}{|\nabla f|}dA.
$$
By the divergence theorem we have
$$
2KV(r)\geq 2\int_{D(r)}\Delta f=2\int_{\partial D(r)}|\nabla f|dA.
$$
By (\ref{cao2}) we know that on $\partial D(r)$, there is a constant $C>2$ such that for $r\geq 2C$,
$$
|\nabla f|^2\geq f-C.
$$
Then we have
$$
2\int_{\partial D(r)}|\nabla f|dA\geq  2\int_{\partial D(r)}\frac{f-C}{|\nabla f|}dA,
$$
The right side of above inequality is
$$
\geq (r-2)V'(r).
$$
Hence we have
$$
2KV(r)\geq (r-2)V'(r),
$$
which then implies that
$$
V(r)\leq V(2C)r^{2K}
$$
for $r>2C$.
\end{proof}
We remark that the above argument is motivated from the proof of the volume growth estimate in \cite{CZ}.

\section{finite energy harmonic functions on steady Ricci solitons}\label{sect4}

Let $(M,g)$ be a complete non-compact Riemannian manifold of dimension $n$.
Fix $p\in M$. In this section we always assume that $(M,g)$ satisfies
$Ric_f\geq \lambda g$ for some constant $\lambda\geq 0$ with nonnegative scalar curvature, i.e., $R\geq 0$ and $|\nabla f|\leq\alpha d(x,p)+b$. Then we have $R+\Delta f\geq n\lambda$ on $M$.
We study the $L^2$ estimate for hessian matrix for harmonic functions with finite energy.

\begin{proof} (of Proposition \ref{wangjp2}).
Let $u:M\to R$ be a harmonic function on $(M,g,f)$ with finite energy
$$
\int_M|\nabla u|^2<\infty.
$$
Recall that the Bochner formula for the harmonic function $u:M\to R$,
$$
\frac{1}{2}\Delta |\nabla u|^2=|\nabla^2 u|^2+Rc(\nabla u,\nabla u).
$$
Using the assumption $Ric_f\geq \lambda g$ we have
\begin{equation}\label{bochner}
\frac{1}{2}\Delta |\nabla u|^2\geq |\nabla^2 u|^2+\lambda|\nabla u|^2-\nabla^2f(\nabla u,\nabla u).
\end{equation}
Recall that the Hessian matrix $\nabla^2f=(f_{ij})$ in local coordinates $(x_i)$ in $M$.

Let $\phi=\phi_R$ be the cut-off function on $B_{2R}(p)$. We write by $o(1)$ the quantities such that $o(1)\to 0$ as $R\to\infty$.
Then, we have
$$
\int_M(|\nabla^2 u|^2+\lambda|\nabla u|^2)\phi^2=\int_M(\frac{1}{2}\Delta |\nabla u|^2+\nabla^2f(\nabla u,\nabla u))\phi^2.
$$
By direct computation we have
$$
\int_M(\frac{1}{2}\Delta |\nabla u|^2)\phi^2=\int_M(\frac{1}{2} |\nabla u|^2)\Delta\phi^2=o(1)
$$
and
$$
\int_Mf_{ij}u_iu_j\phi^2=\frac{1}{2}\int_M \Delta f |\nabla u|^2\phi^2-\int f_i\phi_i|\nabla u|^2\phi+o(1).
$$
Then we have
$$
\int_Mf_{ij}u_iu_j\phi^2=\frac{1}{2}\int_M (n\lambda -R) |\nabla u|^2\phi^2+o(1).
$$
Hence by (\ref{bochner}) we have
$$
\int_M(|\nabla^2 u|^2+\lambda|\nabla u|^2)\phi^2+\frac{1}{2}\int_M R|\nabla u|^2\phi^2\leq \frac{n\lambda}{2}\int_M|\nabla u|^2\phi^2+o(1).
$$
Sending $R\to\infty$ we obtain that
\begin{equation}\label{bochner2}
\int_M|\nabla^2 u|^2+\frac{1}{2}\int_M R|\nabla u|^2\leq \frac{\lambda(n-2)}{2}\int_M|\nabla u|^2.
\end{equation}
\end{proof}

We now give application of this integral inequality. The following result is well-known, but we include it by a direct application of our integral inequality.

\begin{Pro} When $Ric_f=\lambda g$ with $\lambda>0$ and $n=2$. If $u$ is a finite energy harmonic function on $M$, then it is a constant.
\end{Pro}
\begin{proof} By the integral inequality, we have $\nabla^2u=0$ and $R|\nabla u|^2=0$. Then either $R=0$ or $R>0$ and $u$ is a constant function on $M$.
If $R=0$, then by $\nabla^2f=\lambda g$, we know that $(M,g)$ is a warped product and it is the Gaussian soliton on $R^2$. In this case by the Liouville theorem we know that $u$ is a constant.
\end{proof}

The importance of the integral estimate is the following. We consider the case when $\lambda=0$ and $R\geq 0$.

\begin{Pro}\label{wang16}
Assume $(M,g,f)$ satisfies $Ric_f\geq 0$ on $M$ with $R>0$. Then there is no nontrivial harmonic function on $(M,g)$ with finite energy.
\end{Pro}
\begin{proof}
We argue by contradiction.
Assume that there is a nontrivial harmonic function with finite energy on $(M,g)$.
By (\ref{bochner2}) we know that
$$
\int_M|\nabla^2 u|^2+\frac{1}{2}\int_M R|\nabla u|^2=0.
$$
Hence $\nabla u$ is a parallel vector field on $M$ and $R=0$, a contradiction with $R>0$.
\end{proof}

This result slightly generalizes a Liouville type theorem (Theorem 4.1 in \cite{MS}) on a gradient steady Ricci soliton.
For completeness we carry on the argument above to the case when $(M,g)$ is a gradient steady Ricci soliton, which gives a new proof of a result due to Munteanu-Sesum \cite{MS} that there is no nontrivial harmonic function with finite energy on the steady Ricci sliton $(M,g)$. In fact, assume $(M, g,f)$ is a nontrivial steady Ricci soliton.
Recall that it is well-known that either $R>0$ or $R=0$ on $M$. By (\ref{wang16}), we have $R=0$, and then the equation
$$
\Delta_f R=-2|Ric|^2,
$$
we know that $Ric=0$ on $M$. By Yau's result \cite{SY}, we know that there is no nontrivial harmonic function with finite energy.

The argument above also shows that if $(M,g,f)$ satisfies $Ric_f\geq 0$ on $M$ with $R\geq0$ and there is a nontrivial harmonic function with finite energy on $M$, then $(M,g)$ is scalar-flat, i.e., $R=0$ and $\Delta f\geq 0$ on $M$.


\begin{thebibliography}{30}

 \bibitem{BE}D. Bakry and M. Emery, \emph{Diffusions hypercontractives}, Seminaire de probabilites, XIX, 1983/84, volume 1123 of Lecture Notes in Math., 177-206. Springer, Berlin, 1985.


\bibitem{C1} H.D. Cao, \emph{Recent progress on Ricci solitons}, Adv. Lect. Math. 11 (2) (2010), 1-38.

\bibitem{C2} X. Cao, \emph{Compact gradient shrinking Ricci solitons with positive curvature operator}, Journal of Geometric Analysis, 17, 451-459, (2007).

\bibitem{C} B.L. Chen, \emph{Strong uniqueness of the Ricci flow}, J. Differential Geom. 82 (2009), no. 2, 362-382.

  \bibitem{CG} J. Cheeger and D. Gromoll, \emph{The splitting theorem for manifolds of nonnegative Ricci curvature}, J. Differential Geometry 6 (1971), 119-128.


 \bibitem{CZ} H.D. Cao and D. Zhou, \emph{On complete gradient shrinking Ricci solitons}, J. Differential Geom. 85 (2010), no. 2, 175-186.

 \bibitem{CN} J. Carillo and L. Ni,\emph{ Sharp logarithmic sobolev inequalities on gradient solitons and applications}. Comm. Anal. Geom. 17, 721-753 (2009).

 \bibitem{CLN}B. Chow, P. Lu and L. Ni, \emph{Hamilton's Ricci flow}, Graduate studies in mathematics, 2006.


 \bibitem{H} R. Hamilton, \emph{The formation of singularities in the Ricci flow}, Surveys in Differential Geom. 2 (1995), 7-136, International Press.

 \bibitem{HM} R. Haslhofer and R. Muller, \emph{A compactness theorem for complete Ricci shrinkers}, Geom. Funct. Anal. 21 (2011), 1091-1116.

 \bibitem{L}P. Li,\emph{ Lecture Notes on Geometric Analysis}, Lecture Notes Series No. 6, Research Institute of Mathematics, Global Analysis Research Center, Seoul National University, Korea (1993).

\bibitem{Lo} J. Lott, \emph{Some geometric properties of the Bakry-Emery-Ricci tensor}, Comm. Math. Helv. 78 (2003), 865-883

\bibitem{LV} J. Lott and C. Villani, \emph{Ricci curvature for metric-measure spaces via optimal transport}, Annals of Math. 169 (2009), 903-991.


\bibitem{L} A. Lichnerowicz, \emph{Varietes riemanniennes a tensor C non negatif}. C. R. Acad. Sci. Paris Sr. A, 271 (1970) 650-653.

\bibitem{MD}Li Ma, Shenghua Du, \emph{Extension of Reilly formula with applications to eigenvalue estimates for drifting Laplacians}, C. R. Acad. Sci. Paris, Ser. I 348 (2010) 1203-1206

\bibitem{Ma} Li Ma,  \emph{Remarks on compact shrinking Ricci solitons of dimension four}, C. R. Acad. Sci. Paris, Ser. I 351 (2013) 817-823

\bibitem{M} F. Morgan, \emph{Manifolds with Density}. Notices of the Amer. Math. Soc., 52 (2005), no. 8, 853-858.

\bibitem{MS} O. Munteanu and N. Sesum, \emph{On gradient Ricci solitons}, J. Geom. Anal.,(2013)23:539-561

\bibitem{MW1} O.Munteanu, J.P.Wang, \emph{ Geometry of shrinking Ricci soliton}
Compositio Math., 151(2015)2273-2300.

\bibitem{MW2}
Ovidiu Munteanu, Jiaping Wang, \emph{Geometry of manifolds with
densities},Adv. Math. 259(2014)269-305

\bibitem{MW3} O. Munteanu and J. Wang, \emph{Smooth metric measure spaces with nonnegative curvature}, Comm. Anal. Geom 19 (2011), no. 3, 451-486.

\bibitem{MW4} O. Munteanu and J. Wang,\emph{ Analysis of the weighted Laplacian and applications to Ricci solitons}, Comm. Anal. Geom. 20 (2012), no. 1, 55-94.


\bibitem{N}
A.Naber, \emph{Noncompact shrinking four solitons with nonnegative curvature
}, J. reine angew. Math. 645 (2010), 125-153 DOI 10.1515/CRELLE.2010.062


\bibitem{P} G. Perelman, \emph{The entropy formula for the Ricci flow and its geometric applications}, arXiv:math. DG/0211159.

\bibitem{SY} R. Schoen and S.T. Yau, Lectures on Differential Geometry, International Press, 1994.

\bibitem{S1} K. Sturm, \emph{On the geometry of metric measure spaces}. I. Acta Math. 196 (2006), 65-131.

\bibitem{S2} K. Sturm, \emph{On the geometry of metric measure spaces}. II. Acta Math. 196 (2006), 133-177.

\bibitem{WW} G. Wei and W. Wylie, \emph{Comparison geometry for the Bakry-Emery Ricci tensor}, J. Differential Geom (2009), 377-405.

\bibitem{Y} N. Yang, \emph{A note on nonnegative Bakry-Emery Ricci Curvature}, Arch. Math. 93 (2009), no. 5, 491-496.



\end{thebibliography}
\end{document}